\documentclass{amsart}

\usepackage{macros}
\standardsettings
\colorcommentstrue

\usepackage[top=1in, left=1in, right=1in, bottom=1in]{geometry}

\draftfalse

\begin{document}

\authorlior\authorvanessa\authordavid

\title{The Banach--Mazur--Schmidt and Banach--Mazur--McMullen games}

\maketitle

\begin{Abstract}

We introduce two new mathematical games, the Banach--Mazur--Schmidt game and the Banach--Mazur--McMullen game, merging well-known games. We investigate the properties of the games, as well as providing an application to Diophantine approximation theory, analyzing the geometric structure of certain Diophantine sets.

\end{Abstract}

\section{Introduction}

\subsection{Schmidt's game and the Banach--Mazur game}
\label{subsectiongames}

The Banach--Mazur game, dating back to 1935, is arguably the prototype for all infinite mathematical games. This game has been extensively studied and we refer the interested reader to \cite{Telgarsky, Revalski} for a thorough historical overview and recent developments. One of the most interesting aspects of the game is its connection to topology, namely that one of the players has a winning strategy if and only if the target set is comeager.

In 1966, W. M. Schmidt \cite{Schmidt1} introduced a two-player game referred to thereafter as Schmidt's game. This game may be considered in a sense as a variant of the Banach--Mazur game. Schmidt invented the game primarily as a tool for studying certain sets which arise in number theory and Diophantine approximation theory. These sets are often exceptional with respect to both measure and category. The most significant example is the following. Let $\Q$ denote the set of rational numbers. A real number $x$ is said to be \emph{badly approximable} if there exists a positive constant $c = c(x)$ such that $\left|x-\frac{p}{q}\right| > \frac{c}{q^2}$ for all $\frac{p}{q}\in \Q$. We denote the set of badly approximable numbers by $\BA$. This set plays a major role in Diophantine approximation theory, and is well-known to be both meager and Lebesgue null. Nonetheless, using his game, Schmidt was able to prove the following remarkable result:

\begin{theorem}[Schmidt \cite{Schmidt1}]
Let $(f_n)_{n=1}^{\infty}$ be a sequence of $\CC^1$ diffeomorphisms of $\R$. Then the Hausdorff dimension of the set $\bigcap_{n=1}^{\infty}f^{-1}_n(\BA)$ is $1$. In particular, $\bigcap_{n=1}^{\infty}f^{-1}_n(\BA)$ is uncountable.
\end{theorem}

\begin{remark}
We shall describe the games in the context of complete metric spaces. One could consider a more general framework of topological games, but as all of our applications and results are in this more restricted context, we prefer not to follow the most general presentation.
\end{remark}

\subsection{Description of games}
Let $(X,\dist)$ be a complete metric space. In what follows, we denote by $B(x,r)$ and $B^\circ(x,r)$ the closed and open balls in the metric space $(X,\dist)$ centered at $x$ of radius $r$, i.e.,
\begin{equation}
\label{ballsdef}
B(x,r) \df \{ y \in X : d(x,y) \leq r\}, \;\; B^\circ(x,r) \df \{ y \in X : d(x,y) < r\}.
\end{equation}
Let $\Omega \df X \times \R_+$ be the set of formal balls in $X$, and define a partial ordering on $\Omega$ by letting
\[
(x_2,r_2)\leq_s(x_1,r_1) \text{ if } r_2+d(x_1,x_2)\le r_1.
\]
We associate to each pair $(x,r)$ a closed ball and an open ball in $(X,\dist)$ via the `ball' functions $B(\cdot ,\cdot)$ and $B^\circ(\cdot,\cdot)$ as in (\ref{ballsdef}).
Note that the inequality $(x_2,r_2)\leq_s (x_1,r_1)$ clearly implies (but is not necessarily implied by\Footnote{For example, if $X = \{x_0\}$ is a singleton then $B(x_0,r_2)\subset B(x_0,r_1)$ for all $r_1,r_2 > 0$, but the inequality $(x_0,r_2)\leq_s (x_0,r_1)$ only holds if $r_2 \leq r_1$.}) the inclusion $B(x_2,r_2) \subset B(x_1,r_1)$. Nevertheless, the two conditions are equivalent when $(X,\dist)$ is a Banach space.

For Schmidt's game, fix $\alpha,\beta\in (0,1)$ and $S\subset X$. The set $S$ will be called the \emph{target set}. Schmidt's $(\alpha,\beta,S)$-game is played by two players, whom we shall call Alice and Bob. The game starts with Bob choosing a pair $\omega_1 = (x_1,r_1) \in \Omega$. Alice and Bob then take turns choosing pairs $\omega'_n = (x'_n,r'_n)\leq_s\omega_n$ and $\omega_{n+1}= (x_{n+1},r_{n+1})\leq_s\omega'_n$, respectively. These pairs are required to satisfy
\begin{equation}
\label{Schmidt_rules}
r_n' = \alpha r_n\text{ and }r_{n+1} = \beta r_n'\,.
\end{equation}
Since the game is played on a complete metric space and since the diameters of the nested balls
\begin{equation}
\label{nestedballs}
B(\omega_1) \supset \ldots\supset B(\omega_n) \supset B(\omega'_n) \supset B(\omega_{n + 1}) \supset \ldots
\end{equation}
tend to zero as $n\rightarrow\infty$, the intersection of these balls is a singleton $\{x_\infty\}$. Call Alice the winner if $x_\infty\in S$; otherwise Bob is declared the winner. A \emph{strategy} consists of a description of how one of the players should act based on the opponent's previous moves. A strategy is \emph{winning} if it guarantees the player a win regardless of the opponent's moves. If Alice has a winning strategy for Schmidt's $(\alpha,\beta,S)$-game, we say that $S$ is an $(\alpha,\beta)$-\emph{winning} set. If $S$ is $(\alpha,\beta)$-winning for all (equiv. for all sufficiently small) $\beta\in (0,1)$, we say that $S$ is an \emph{$\alpha$-winning} set. If $S$ is $\alpha$-winning for some (equiv. for all sufficiently small) $\alpha\in (0,1)$, we say that $S$ is \emph{winning}. (To see that ``for all'' and ``for some'' may be replaced by ``for all sufficiently small'', cf. \cite[Lemmas 8 and 9]{Schmidt1}.)

In what follows we shall need a variation of Schmidt's game introduced by C. T. McMullen \cite{McMullen_absolute_winning}, the \emph{absolute winning} game.\Footnote{Technically, the game we describe below was not defined by McMullen, who only considered the special case $X = \R^d$. The version of the absolute winning game described below appeared first in \cite[\64]{FSU4}.} 
Given $\beta\in (0,1)$ and $S\subset X$, the \emph{$(\beta,S)$-absolute game} is played as follows: As before the game starts with Bob choosing a pair $\omega_1 = (x_1,r_1) \in \Omega$, and Alice and Bob then take turns choosing pairs $\omega_n$ and $\omega'_n$.  However, instead of requiring $\omega_{n + 1} \leq_s \omega'_n \leq_s \omega_n$, now there is no restriction on Alice's choice $\omega'_n = (x'_n,r'_n)$, and Bob's choice $\omega_{n + 1} = (x_{n + 1},r_{n + 1})$ must be chosen to satisfy
\begin{equation}
\label{absolutewinningcontainment}
\omega_{n + 1} \leq_s \omega_n \text{ and } \dist(x'_n,x_{n + 1}) \geq r_n' + r_{n + 1}.
\end{equation}
The second condition states that the balls $\omega'_n$ and $\omega_{n + 1}$ are ``formally disjoint'', so we can think of Alice has having ``deleted'' the ball $B(\omega'_n)$. The restrictions on the radii in the absolute game are
\begin{equation}
\label{absolute_winning_rules}
r_n' \leq \beta r_n\text{ and }r_{n+1} \geq \beta r_n.
\end{equation}
However, since \eqref{absolute_winning_rules} is insufficient to ensure that the diameters of the nested balls \eqref{nestedballs} tend to zero, it may happen that the intersection $I \df \bigcap_n B(\omega_n)$ is not a singleton. If this occurs, we call Alice the winner if $I\cap S\neq\emptyset$; otherwise Bob is declared the winner. It may also happen that Bob has no legal moves, and for technical reasons in this case it is better to declare Alice the winner (cf. \cite[comments before Lemma 2.1]{BFS1}). However, for sufficiently nice spaces (i.e. uniformly perfect spaces; cf. \cite[Lemma 4.3]{FSU4}) and for sufficiently small $\beta$, such a situation cannot arise. If Alice has a winning strategy for the $\beta$-absolute game with a given target set $S$, then $S$ is called \emph{$\beta$-absolute winning}, and if this is true for every $\beta > 0$, then $S$ is called \emph{absolute winning}. Every absolute winning set on a uniformly perfect set is winning \cite[Proposition 4.4(ii)]{FSU4}.

The Banach--Mazur game's rules are the same as for Schmidt's game except for the fact that no restricting parameters are given, i.e., at each of the player's turns, they may choose as small a radius as they please, and just like in the absolute winning game, if the intersection of the players' balls is not a singleton, we declare Alice the winner if this intersection with the target set is nonempty. It is well-known that Alice has a winning strategy if and only if the target set is comeager \cite{Oxtoby2}.

{\bf Acknowledgements.} The first-named author was supported in part by the Simons Foundation grant \#245708. The third-named author was supported in part by the EPSRC Programme Grant EP/J018260/1. The authors thank the anonymous referee for helpful comments.

\section{The Banach--Mazur--Schmidt and Banach--Mazur--McMullen games}

We now define two new games: the Banach--Mazur--Schmidt (BMS) game and the Banach--Mazur--McMullen (BMM) game. In the BMS (resp. BMM) game, Bob starts, playing according to the Banach--Mazur game rules, while Alice is dealt a parameter $\beta \in (0,1)$ and follows the rules for Schmidt's game (resp. the absolute winning game). More precisely: in the $\beta$-BMS game, Bob and Alice take turns choosing pairs $\omega_n$ and $\omega'_n$ satisfying $\omega_{n + 1} \leq_s \omega'_n \leq_s \omega_n$, while Alice's choices are additionally required to satisfy \eqref{Schmidt_rules}. And in the $\beta$-BMM game, Bob and Alice take turns choosing pairs $\omega_n$ and $\omega'_n$ satisfying \eqref{absolutewinningcontainment}, but Bob's moves are not required to satisfy \eqref{absolute_winning_rules} even though Alice's are.

If Alice has a winning strategy for the $\beta$-BMS (resp. $\beta$-BMM) game, then we call the target set \emph{$\beta$-BMS (resp. $\beta$-BMM) winning}, and if a set it $\beta$-BMS (resp. $\beta$-BMM) winning for all sufficiently small $\beta \in (0,1)$, then we call it \emph{BMS-winning (resp. BMM-winning)}.

Our first theorem geometrically characterizes the $\beta$-BMS and $\beta$-BMM winning sets, but first we need the following definition:

\begin{definition}
\label{definitionporous}
Fix $\beta > 0$. A set $E \subset X$ is said to be \emph{uniformly $\beta$-porous} if there exists $r_0 > 0$ such that for every ball $B(x,r)\subset X$ with $r \leq r_0$, there exists $B^\circ(y,\beta r) \subset B(x,r)$ such that $B^\circ(y,\beta r) \cap E = \emptyset$.
\end{definition}

\begin{theorem}
\label{theoremcharacterization}
Let $(X,d)$ be a separable complete metric space and fix $\beta \in (0,1)$. Then a Borel set $T \subset X$ is $\beta$-BMS winning if and only if $X\butnot T$ can be written as the countable union of uniformly $\beta$-porous sets. Moreover, $T$ is $\beta$-BMM winning if and only if $X\butnot T$ is countable.
\end{theorem}

\begin{example}
The Cantor set $C\subset\R$ is uniformly $1/5$-porous,\Footnote{Let $B(x,r)\subset\R$ be a ball, and we will show that there exists $B^\circ(y,r/5)\subset B(x,r)$ such that $B^\circ(y,r/5) \cap C = \emptyset$. By a zooming argument, we can without loss of generality assume that $B(x,r)\cap [0,1/3]\neq \emptyset$ and $B(x,r)\cap [2/3,1]\neq \emptyset$. By a symmetry argument, we can without loss of generality assume that $x\geq 1/2$. If $r/5\leq 1/6$, then we let $y = 1/2$, and if $3r/5 \geq 1/2$, then we let $y = 1/2 + 4r/5$. Either way we get $B^\circ(y,r/5)\subset [1/3,x + r] \subset B(x,r)$ and $B^\circ(y,r/5) \cap C = \emptyset$.} so by Theorem \ref{theoremcharacterization}, $\R\butnot C$ is $1/5$-BMS winning.
\end{example}

A slightly more general example:

\begin{example}
Given $s\geq 0$, a closed set $K$ is called \emph{Ahlfors $s$-regular} if there exists a measure $\mu$ whose support equals $K$ and a constant $C > 0$ such that for all $x\in K$ and $0 < r \leq 1$,
\[
C^{-1} r^s \leq \mu(B(x,r)) \leq C r^s.
\]
If $K\subset \R^d$ is an Ahlfors $s$-regular set with $s < d$, or more generally if $K\subset X$ is Ahlfors $s$-regular, $X$ is Ahlfors $\delta$-regular, and $s < \delta$, then a simple calculation shows that $K$ is uniformly porous,\Footnote{Suppose otherwise, and let $\beta > 0$ be small. Then there exists a ball $B(x,r)$ such that every ball $B^\circ(y,\beta r) \subset B(x,r)$ intersects $K$. Let $A\subset B(x,r/2)$ be a maximal $3\beta r$-separated set. Since $X$ is Ahlfors $\delta$-regular, $\#(A) \geq C_1 \beta^{-\delta}$ for some constant $C_1 > 0$. For each $x\in A$ let $f(x)\in K$ be chosen so that $f(x)\in B^\circ(x,\beta r)$. Then $B = \{f(x) : x\in A\}$ is a $\beta r$-separated subset of $B(x,r)\cap K$, so since $K$ is Ahlfors $s$-regular, we have $\#(B) \leq C_2 \beta^{-s}$ for some constant $C_2 > 0$. Since $\#(A) = \#(B)$, this is a contradiction for sufficiently small $\beta$. Thus $K$ is uniformly porous. An alternate proof of this fact may be found in \cite[Lemma 3.12]{BHR}.} so by Theorem \ref{theoremcharacterization}, $T = X\butnot K$ is BMS-winning.
\end{example}

If $X$ is Ahlfors $\delta$-regular, then for every $\beta$ there exists $s_\beta < \delta$ such that every uniformly $\beta$-porous set $T$ has upper box-counting dimension $\leq s_\beta$ \cite[Theorem 4.7]{JJKRRS}. Since the Hausdorff and packing dimensions of a set are bounded above by its upper box dimension, we get the following corollary of Theorem \ref{theoremcharacterization}:

\begin{corollary}
If $X$ is Ahlfors $\delta$-regular and $T\subset X$ is BMS-winning, then the Hausdorff and packing dimensions of $X\butnot T$ are $<\delta$.
\end{corollary}

The following can be proven either using Theorem \ref{theoremcharacterization} or by a method similar to \cite[Theorem 2]{Schmidt1}.

\begin{corollary}
The intersection of countably many $\beta$-BMS (resp. $\beta$-BMM) winning sets is $\beta$-BMS (resp. $\beta$-BMM) winning.
\end{corollary}
~

\begin{proof}[Proof of Theorem \ref{theoremcharacterization}]~

($\Rightarrow$): Suppose Alice has a winning strategy. By \cite[Theorem 7]{Schmidt1}, she has a \emph{positional} winning strategy, i.e. a map $f$ which inputs a move of Bob and tells her what she should do next. Let $\BB$ denote the set of balls in $X$, so that $f:\BB\to\BB$ and if $f(B(x,r)) = B(y,s)$, then $s = \beta r$. Let
\[
g(B(x,r)) = \begin{cases}
B^\circ(y,s) & \text{BMS game}\\
B^\circ(x,r)\butnot B(y,s) & \text{BMM game}
\end{cases}.
\]
For each $m\in\N$, let
\[
K_m = X \butnot \bigcup_{\substack{B = B(x,r)\in\BB \\ 0 < r \leq 1/m}} g(B).
\]
\begin{claim}
$X\butnot T\subset \bigcup_{m\in\N} K_m$.
\end{claim}
\begin{subproof}
By contradiction, suppose $p\in (X\butnot T)\butnot\bigcup_{m\in\N} K_m$.
We claim that Bob can beat Alice's strategy by using the following counter-strategy: always choose a ball $B\in\BB$ such that $p\in g(B)$.
Obviously, if he can successfully apply this strategy then this is a contradiction, since then the intersection point will be $p\in X\butnot T$, a win for Bob, but Alice's strategy was supposed to be a winning strategy.
We prove by induction that he can apply the strategy.
If he applied it to choose his previous move $B_n$, then $p\in g(B_n)$, and from the definition of $g$, this guarantees the existence of a neighborhood $B(p,2/m)$ of $p$ such that any ball contained in $B(p,2/m)$ constitutes a legal move for Bob.
Such a neighborhood also exists if it is the first turn and no one has made a move.
Now since $p\notin K_m$, there exists $B = B(x,r)\in\BB$ with $0 < r \leq 1/m$ and $p\in g(B)$.
Then $B$ constitutes a legal move for Bob, since $p\in g(B) \subset B$ and thus $B \subset B(p,2/m)$.
\end{subproof}

Now if Alice and Bob are playing the BMS game, then for every $B = B(x,r)\in\BB$ such that $0 < r\leq 1/m$, we have $g(B) \cap K_m = \emptyset$, so by definition, $K_m$ is uniformly $\beta$-porous, completing the proof.

On the other hand, suppose that Alice and Bob are playing the BMM game.
Then for every  $B = B(x,r)\in\BB$ such that $0 < r\leq \beta^{1/2}/m$, we have $g(B(x,\beta^{-1/2}r)) \cap K_m = \emptyset$, so $B\cap K_m = B(x_2,\beta^{1/2}r)\cap K_m$ for some $x_2\in X$. Continuing this process we get a sequence $(x_k)_1^\infty$ with $x_1 = x$ such that
\[
B\cap K_m = B(x_2,\beta^{1/2}r)\cap K_m = B(x_3,\beta r)\cap K_m = \cdots
\]
So $\diam(B\cap K_m) = 0$ and thus $B\cap K_m$ is either empty or a singleton. Since $X$ is separable, this implies that $K_m$ is countable, completing the proof.

($\Leftarrow$): Suppose that Alice and Bob are playing the BMS game and that $X\butnot T = \bigcup_1^\infty E_n$, where each $E_n$ is uniformly $\beta$-porous. For each $n$, Alice can avoid the set $E_n$ in a finite number of moves as follows: make dummy moves until Bob's radius is smaller than the $r_0$ which occurs in Definition \ref{definitionporous}, then make the move $B(y,\beta r) \subset B(x,r)$, where $x$, $y$, and $r$ are as in Definition \ref{definitionporous}, then make one more move to avoid the set $B(y,\beta r)\butnot B^\circ(y,\beta r)$. By avoiding each set $E_n$ in turn, Alice can ensure that the intersection point is in $T$.

On the other hand, suppose that Alice and Bob are playing the BMM game and that $X\butnot T$ is countable. If $(x'_n)_1^\infty$ is an enumeration of $X\butnot T$, then let Alice's $n$th move be $\omega'_n = (x'_n,r'_n)$ for some legal $r'_n$. This ensures that the intersection point is in $T$.
\end{proof}

\section{Application to Diophantine approximation}

Recall that the \emph{exponent of irrationality} of a vector $\xx\in\R^d$ is the number
\[
\omega(\xx) = \limsup_{\pp/q\in\Q^d} \frac{-\log\|\xx - \pp/q\|}{\log(q)},
\]
where the limsup is taken along any enumeration of $\Q^d$. The set
\[
\{\xx\in\R^d : \omega(\xx) = 1 + 1/d\}
\]
is of full Lebesgue measure and is winning for Schmidt's game, while the set
\[
\{\xx\in\R^d : \omega(\xx) = \infty\}
\]
is comeager, so it is winning for the Banach--Mazur game. A natural question is whether their union is winning for the BMS game. The following result shows that the answer is no:

\begin{theorem}
Let $\psi:\N\to (0,\infty)$ be a decreasing function such that $q^{1 + 1/d}\psi(q) \to 0$, and let
\begin{equation}
\label{liminf}
S = \left\{\xx\in\R^d : 0 < \liminf_{\pp/q\in\Q^d} \frac{\|\xx - \pp/q\|}{\psi(q)} < \infty\right\}.
\end{equation}
Then for every $\beta$, Bob has a strategy to ensure that the intersection point is in $S$. In particular, $X\butnot S$ is not BMS-winning.
\end{theorem}
\begin{corollary}
The set \eqref{liminf} cannot be written as the union of countably many uniformly $\beta$-porous sets for any $0 < \beta < 1$.
\end{corollary}
Note that for any $c > 1 + 1/d$, the set $\{\xx\in\R^d : \omega(\xx) = c\}$ contains a set of the form \eqref{liminf}, so it also cannot be written as the union of countably many uniformly $\beta$-porous sets.
\begin{proof}
Fix $0 < \beta < 1$. For each $\pp/q\in\Q^d$, write
\[
B(\pp/q) = B(\pp/q,\psi(q)), \;\; B'(\pp/q) = B(\pp/q,(1 + 6\beta^{-1})\psi(q)).
\]
We will give a strategy for Bob to force the intersection point to lie in infinitely many of the sets $B'(\pp/q)$, but only finitely many of the sets $B(\pp/q)$. Accordingly, we fix $Q_0\in\N$ large to be determined, and we call a ball $A = B(\xx,r)\subset\R^d$ \emph{good} if for every $\pp/q\in\Q^d$ such that $A\cap B(\pp/q)\neq\emptyset$ and $q\geq Q_0$, we have
\begin{equation}
\label{psiqbound}
\psi(q) \leq r/3.
\end{equation}
Intuitively, if $A$ is a good ball then Bob should still be able to win and avoid all of the sets $B(\pp/q)$, after Alice has just played $A$.

\begin{claim}
\label{claimgood}
If $A = B(\xx,r)$ is a good ball, then there exists a ball $B = B(\yy,s)\subset A$ such that $B\subset B'(\pp_0/q_0)$ for some $\pp_0/q_0\in\Q^d$ with $\psi(q_0) < r$, and such that for every $\pp/q\in\Q^d$ such that $B\cap B(\pp/q)\neq\emptyset$ and $q\geq Q_0$, we have
\[
\psi(q) \leq \beta s/3.
\]
\end{claim}
In other words, if Alice's previous choice was good, then Bob can move so that Alice's next choice must be good, while at the same time moving sufficiently close to a rational point.
\begin{subproof}
By the Simplex Lemma \cite[Lemma 4]{KTV}, there exists an affine hyperplane $\LL\subset\R^d$ such that for all $\pp/q\in\Q^d\cap B(\xx,2r)\butnot\LL$, we have
\begin{equation}
\label{simplex}
q \geq c_d r^{-d/(d + 1)},
\end{equation}
where $c_d > 0$ is a constant depending on $d$. Choose a ball $\w A = B(\w\xx,r/3)\subset A\butnot\thickvar\LL{r/3}$, where $\thickvar\LL t = \{\xx\in\R^d : \dist(\xx,\LL) \leq t\}$ is the closed $t$-thickening of $\LL$. Note that for all $\pp/q\in\Q^d$, if $\w A\cap B(\pp/q)\neq\emptyset$ and $q\geq Q_0$, then $A\cap B(\pp/q)\neq\emptyset$, so by \eqref{psiqbound}, $\dist(\pp/q,\w A) \leq \psi(q) \leq r/3$. Thus $\pp/q\in \thickvar{\w A}{r/3} \subset B(\xx,2r)\butnot \LL$, so \eqref{simplex} holds.

Let $\pp_0/q_0\in\Q^d$ be chosen to minimize $q_0$, subject to the conditions $\w A\cap B(\pp_0/q_0)\neq\emptyset$ and $q_0\geq Q$. Let $s = 3\beta^{-1}\psi(q_0)$. Since $q^{1 + 1/d} \psi(q) \to 0$, if $Q_0$ is sufficiently large then \eqref{simplex} implies $s < r/3$. Thus there exists a ball $B = B(\yy,s) \subset \w B$ such that $B\cap B(\pp_0/q_0) \neq \emptyset$ and thus $B\subset B'(\pp_0/q_0)$. Now if $\pp/q\in\Q^d$ satisfies $B\cap B(\pp/q)\neq\emptyset$ and $q\geq Q_0$, then $q\geq q_0$, so
\[
\psi(q) \leq \psi(q_0) = \beta s/3.
\varqedhere\]
\end{subproof}
By choosing $Q_0$ sufficiently large, we can guarantee that the ball $B(\0,1)$ is good. Let Bob's strategy consist of responding to Alice's moves $A$ with the balls $B$ given in Claim \ref{claimgood}, letting $A = B(\0,1)$ for the first move. Then by induction, Alice's moves will always be good, which implies that the intersection point $\zz$ is not contained in any of the balls $B(\pp/q)$. But by construction, $\zz$ is contained in infinitely many balls $B(\pp_0/q_0)$. Thus $\zz\in S$, where $S$ is as in \eqref{liminf}.
\end{proof}

A natural point of comparison for the exponent of irrationality function is the \emph{Lagrange spectrum function}
\[
L(\xx) = \liminf_{\pp/q\in\Q^d} \frac{\|\xx - \pp/q\|}{q^{1 + 1/d}}\cdot
\]
While the condition $\omega(\xx) > 1 + 1/d$ is equivalent to $\xx$'s being very well approximable, the condition $L(\xx) > 0$ is equivalent to $\xx$'s being badly approximable. We have shown above that the levelsets of the exponent of irrationality function cannot be written as the countable union of uniformly $\beta$-porous sets for any $0 < \beta < 1$. To contrast this we prove:

\begin{theorem}
For all $d\in\N$ and $0 < \epsilon < 1$, the set $\WA_d(\epsilon) := \{\xx\in\R^d : L(\xx) \leq \epsilon\}$ is $\beta$-BMS winning, where $\beta = (\epsilon/3)^{d + 1} \in (0,1/2)$.
\end{theorem}
\begin{corollary}
The set $\BA_d(\epsilon) = \R^d\butnot\WA_d(\epsilon)$ can be written as the union of countably many uniformly $\beta$-porous sets.
\end{corollary}
\begin{proof}
Alice's strategy will be as follows: move near a rational point $\pp_n/q_n$, then make a move disjoint from $\pp_n/q_n$, then wait long enough so that Bob's move $B_n = B(\xx_n,r_n)$ satisfies
\[
r_n < 2(3/\epsilon)^3 (q_n \dist(B_n,\pp_n/q_n))^{d + 1},
\]
then repeat. So suppose that Bob has just made the move $B = B(\xx,r) = B(\xx_n,r_n)$, and we will show how Alice can move near a new rational point $\pp_{n + 1}/q_{n + 1}$. Let $Q$ be the unique number such that $r = 2(3/\epsilon)^d/Q^{1 + 1/d}$. By Dirichlet's theorem, there exists $\pp/q = \pp_{n + 1}/q_{n + 1}$ with $q\leq Q$ such that
\begin{equation}
\label{dirichlet}
\left\|\xx - \frac{\pp}{q}\right\| \leq \frac{1}{q Q^{1/d}}\cdot
\end{equation}
Note that this inequality implies that $\pp_{n + 1}/q_{n + 1} \neq \pp_m/q_m$ for all $m\leq n$.

\noindent {\it Case 1.} $\left\|\xx - \frac{\pp}{q}\right\| \leq r/2$. In this case, we have $B(\pp/q,\beta r)\subset B$. On the other hand,
\[
\beta r = \frac{2(3/\epsilon)^d\beta}{Q^{1 + 1/d}} \leq \frac{2(3/\epsilon)^d\beta}{q^{1 + 1/d}} \leq \frac{\epsilon}{q^{1 + 1/d}}\cdot
\]
So the move $B(\pp/q,\beta r)$ will bring Alice sufficiently close to the rational point $\pp/q$.

\noindent {\it Case 2.} $\left\|\xx - \frac{\pp}{q}\right\| \geq r/2$. In this case, by \eqref{dirichlet} we have
\[
\frac{1}{q Q^{1/d}} \geq \frac{r}{2} = \frac{(3/\epsilon)^d}{Q^{1 + 1/d}},
\]
and rearranging gives
\[
q \leq (\epsilon/3)^d Q.
\]
Thus
\[
\frac{r}{2}\leq \left\|\xx - \frac{\pp}{q}\right\| \leq \frac{\epsilon}{3q^{1 + 1/d}},
\]
and in particular
\[
\left\|\xx - \frac{\pp}{q}\right\| + r \leq \frac{\epsilon}{q^{1 + 1/d}}\cdot
\]
It follows that $B\subset B(\pp/q,\epsilon/q^{1 + 1/d})$, so any move Alice makes will bring her sufficiently close to the rational point $\pp/q$.
\end{proof}

\bibliographystyle{amsplain}

\bibliography{bibliography}

\providecommand{\bysame}{\leavevmode\hbox to3em{\hrulefill}\thinspace}
\providecommand{\MR}{\relax\ifhmode\unskip\space\fi MR }
\providecommand{\MRhref}[2]{%
  \href{http://www.ams.org/mathscinet-getitem?mr=#1}{#2}
}
\providecommand{\href}[2]{#2}
\begin{thebibliography}{10}

\bibitem{BHR}
M.~Bonk, J.~Heinonen, and S.~Rohde, \emph{Doubling conformal densities}, J.
  Reine Angew. Math. \textbf{541} (2001), 117--141.

\bibitem{BFS1}
R.~Broderick, L.~Fishman, and D.~S. Simmons, \emph{Badly approximable systems
  of affine forms and incompressibility on fractals}, J. Number Theory
  \textbf{133} (2013), no. 7, 2186--2205.

\bibitem{FSU4}
L.~Fishman, D.~S. Simmons, and M.~Urba\'nski, \emph{{D}iophantine approximation
  and the geometry of limit sets in {G}romov hyperbolic metric spaces},
  \url{http://arxiv.org/abs/1301.5630}, preprint 2013.

\bibitem{JJKRRS}
E.~J{\"a}rvenp{\"a}{\"a}, M.~J{\"a}rvenp{\"a}{\"a}, A.~K{\"a}enm{\"a}ki,
  T.~Rajala, S.~Rogovin, and V.~Suomala, \emph{Packing dimension and {A}hlfors
  regularity of porous sets in metric spaces}, Math. Z. \textbf{266} (2010),
  no.~1, 83--105.

\bibitem{KTV}
S.~Kristensen, R.~Thorn, and S.~L. Velani, \emph{{D}iophantine approximation
  and badly approximable sets}, Advances in Math. \textbf{203} (2006),
  132--169.

\bibitem{McMullen_absolute_winning}
C.~T. McMullen, \emph{Winning sets, quasiconformal maps and {D}iophantine
  approximation}, Geom. Funct. Anal. \textbf{20} (2010), no. 3, 726--740.

\bibitem{Oxtoby2}
J.~C. Oxtoby, \emph{The {B}anach-{M}azur game and {B}anach category theorem},
  Contributions to the theory of games, vol. 3, Annals of Mathematics Studies,
  no. 39, Princeton University Press, Princeton, N.J., 1957, pp.~159--163.

\bibitem{Revalski}
J.~P. Revalski, \emph{The {B}anach-{M}azur game: {H}istory and recent
  developments},
  \url{http://www1.univ-ag.fr/aoc/activite/revalski/Banach-Mazur_Game.pdf}.

\bibitem{Schmidt1}
W.~M. Schmidt, \emph{On badly approximable numbers and certain games}, Trans.
  Amer. Math. Soc. \textbf{123} (1966), 27--50.

\bibitem{Telgarsky}
R.~Telg{\'a}rsky, \emph{Topological games: on the 50th anniversary of the
  {B}anach-{M}azur game}, Rocky Mountain J. Math. \textbf{17} (1987), no.~2,
  227--276.

\end{thebibliography}

\end{document}